\documentclass[11pt,a4paper,reqno]{amsart}

\usepackage{a4wide}
\usepackage[latin1]{inputenc}
\usepackage[T1]{fontenc}
\usepackage{amssymb,amsmath,amsthm,mathrsfs}
\usepackage{xcolor}
\usepackage{graphicx}
\usepackage{float}
\usepackage{enumerate}
\usepackage{upref}
\usepackage{stackrel}
\usepackage{enumitem}
\usepackage[bookmarks=false]{hyperref}

\newtheorem{theorem}{Théorème}

\newtheorem{definition}[theorem]{Definition}
\newtheorem{example}[theorem]{Exemple}

\newtheorem{lemma}[theorem]{Lemma}

\newtheorem{remark}[theorem]{Remark}

\def\rj{ R_{-j}}
\def\Zset{\mathbb{Z}}
\def\Nset{\mathbb{N}}
\def\E{\mathbb{E}}
\def\P{\mathbb{P}}

\def\C{\mathcal{C}}

\begin{document}

\title{Clustering indices and decay of correlations in non-Markovian models}

\author[M. Abadi]{Miguel Abadi}
\address{Miguel Abadi\\ Instituto de Matem\'atica e Estat\'istica\\Universidade de S. Paulo\\ Rua do Mat\~ao 1010\\ Cid. Universitaria\\ 05508090 - São Paulo\\ SP - Brasil} \email{leugim@ime.usp.br}
\urladdr{\url{http://miguelabadi.wixsite.com/miguel-abadi}}

\author[A. C. M. Freitas]{Ana Cristina Moreira Freitas}
\address{Ana Cristina Moreira Freitas\\ Centro de Matem\'{a}tica \&
Faculdade de Economia da Universidade do Porto\\ Rua Dr. Roberto Frias \\
4200-464 Porto\\ Portugal} \email{\href{mailto:amoreira@fep.up.pt}{amoreira@fep.up.pt}}
\urladdr{\url{http://www.fep.up.pt/docentes/amoreira/}}

\author[J. M. Freitas]{Jorge Milhazes Freitas}
\address{Jorge Milhazes Freitas\\ Centro de Matem\'{a}tica \& Faculdade de Ci\^encias da Universidade do Porto\\ Rua do
Campo Alegre 687\\ 4169-007 Porto\\ Portugal}
\email{\href{mailto:jmfreita@fc.up.pt}{jmfreita@fc.up.pt}}
\urladdr{\url{http://www.fc.up.pt/pessoas/jmfreita/}}

\thanks{All authors were partially supported by the joint project FAPESP (SP-Brazil) and FCT (Portugal) with reference FAPESP/19805/2014. ACMF and JMF were partially supported by FCT projects  PTDC/MAT-CAL/3884/2014 and PTDC/MAT-PUR/28177/2017, with national funds, and by CMUP (UID/MAT/00144/2013), which is funded by FCT with national (MCTES) and European structural funds through the programs FEDER, under the partnership agreement PT2020.}

\date{\today}



\begin{abstract}
When there is no independence, abnormal observations may have a tendency to appear in clusters instead of scattered along the time frame. Identifying clusters and estimating their size are important problems arising in statistics of extremes or in the study of quantitative recurrence for dynamical systems. In the classical literature, the Extremal Index appears associated to the cluster size and, in fact, it is usually interpreted as the reciprocal of the mean cluster size. This quantity involves a passage to the limit and in some special cases this interpretation fails due to an escape of mass when computing the limiting point processes. Smith \cite{S88} introduced a regenerative process exhibiting such disagreement. Very recently, in \cite{AFF18} the authors used a dynamical mechanism to emulate the same inadequacy of the usual interpretation of the Extremal Index. Here, we consider a general regenerative  process that includes Smith's model and show that it is important to consider finite time quantities instead of asymptotic ones and compare their different behaviours in relation to the cluster size. We consider other indicators such as what we call the sojourn time, which corresponds to the size of groups of abnormal observations, when there is some uncertainty regarding where the cluster containing that group was actually initiated. We also study the decay of correlations of the non-Markovian models considered.
\end{abstract}

\maketitle

\section{Introduction}

In Extreme Value Theory, the convergence of the maxima of a sequence of i.i.d. random variables is a very well studied subject.
The book \cite{LLR83} is  a major reference on the subject.
The starting point is  that, for $(X_n)_{n \geq \Nset}$ a sequence of i.i.d. random variables over a probability space 
$(\Omega, \mathcal{F}, \mathbb{P})$ with cumulative distribution function $F$,
it is straightforward to see that,
for any real positive $\tau$ and a real sequence $(u_n)_{n \geq 0}$ one has
$$
\underset{n \rightarrow \infty}{\lim} \mathbb{P}(\max\{X_1,...,X_n\} \leq u_n)=e^{-\tau} 
\quad {\rm \ if \ and \ only \ if \ } \quad 
\underset{n \rightarrow \infty}{\lim}  n(1-F(u_n))=\tau .
$$
And the classical three possible limits theorem for the maximum follows.

The independent case is far from modelling the real world, and a major effort to extend this result to dependent
processes has been faced in the last decades.
The principal ingredient is the appearance of the extremal index $\theta$ verifying
$$
\underset{n \rightarrow \infty}{\lim} \mathbb{P}(\max\{X_1,...,X_n\} \leq u_n)=e^{-\theta\tau} 
\ \ {\rm \ whenever \ } \ \ 
\underset{n \rightarrow \infty}{\lim}  n(1-F(u_n))=\tau .
$$
This new factor describes the capacity of a given maximum to produce subsequent ones, due to the correlation of the r.v.'s, ingredient that is absent in the i.i.d. case.
It follows by the above property that the extremal index $\theta \in \left[0 , 1\right]$
and that it is strictly smaller than one for observables that tend to appear in clusters
rather than isolated, as in the i.i.d. case, where it is equal to one.
However, this way of introducing it as a limiting value through the above properties gives rise to certain difficulties to calculate
or even estimate it.

It is the purpose of this paper to show the relevance of considering the extremal index not just as an asymptotic limit
but rather as quantity at finite time, as well as, to show the different behaviours that both cases may present.

It appeared as natural to associate the extremal index with the reciprocal mean of the distribution  of the size of the cluster
of excedentes generated by the correlation of the r.v.'s.. 
The reason is heuristically clear.
Suppose that one wants to observe a cluster of size, at least $k$, of exceedances of the level $u_n$.
That is  
\[\mathbb{P}(N \ge k ) = \mathbb{P}(  \cap_{i=1}^{k} \{ X_i > u_n \}  ) . \]
Here $N$ stands for the number of consecutive observations of the exceedance.
The universal formula
\[
 \mathbb{P}( \cap_{i=1}^{k} A_i  ) = \prod_{i=1}^{k} \mathbb{P}(  A_i | \cap_{j=1}^{i-1} A_j  )
\]
says that, if dependance with respect to the remote past is small and only close past matters,
the factors on the right hand side in the above equality should be all about the same.
If the meaning of "close past" is quantified by looking back up to a distance  $q$,
then  the last display  suggests that 
\begin{equation} \label{heu}
\mathbb{P}(N \ge k ) \approx  \mathbb{P}( \{ X_{q+1} > u_n \}    |  \cap_{i=1}^{q} \{ X_i > u_n \}  )^k , 
\end{equation}
and $N$ has a limiting geometric distribution with success probability 
\begin{equation} \label{entry}
 \mathbb{P}(   X_{q+1} \le u_n |   \cap_{i=1}^{q} \{ X_i > u_n \}  ) , 
\end{equation}
which concludes the intuition.
This heuristic argument was proved to hold under suitable conditions in \cite{A06}.
On the other hand, Aytaç et al. (\cite{AFV15}) constructed several examples where both (limiting) parameters conincide even when
the cluster size distribution has nothing to do with a geometric one.
R.L. Smith proposed an example where this two quantities have different limits \cite{S88}.

In the present paper we have two main purposes. 
Firstly, we want to show that one should consider not only asymptotic limits but also look at the behaviour for finite $n$ in order
to get a full picture of the situation.
Not only because in the real world we only observe finite $n$, but also  because things may behave differently
at finite size and in the limit.
For instance, we will show that, even if, in the limit, the extremal index and the reciprocal of the cluster size are different as in Smith's example, they coincide for finite time.
Also, we show that the above  heuristic argument  may not work   and a subtle different quantity could be more appropriate to consider.
In the above argument, $N$ was considered assuming the \emph{existence} of a cluster, \emph{i.e}, assuming that we started with an exceedance without specifically guaranteeing if that exceedance initiated a cluster, which means that it could correspond to an exceedance inside a cluster initiated in the past.
But we can consider the case where that exceedance is actually \emph{beginning} the cluster.
This would make no difference if the far past is irrelevant. In our models, the extremal index will correspond to the second case and will be different to the first one.

Further, to emphasize the importance of looking at finite and not just limiting statistics, we present another
model where the extremal index does not exist since its asymptotics fluctuate. The same happens with the distribution of the cluster size. However, both can still be identified as the reciprocal of each other for finite observations.

In our case study we also consider the following application.

\vskip0.3cm
{\bf Application: Hitting times.} Parallel to the extreme value theory and totally independently, it was 
deeply studied the theory of hitting times in Poincaré Recurrence Theory.
The review papers \cite{AG04, C00, H13} bring a major panorama of classical results.
Hitting times to balls and cylinder sets were specifically considered.
To fix ideas, consider a sequence $a_0^{n-1}$ and define  the \emph{hitting time}  
$$
\tau_{n}=\inf \left\{  t \geq 1 \ | \  X_t^{t+n-1}=a_0^{n-1} \right\}  .
$$
Now, if an infinite sequence  ${\bf a}=(a_0^{\infty} )$ is fixed, then
one can consider the number of (consecutive) letters of {\bf a} that can be read in the process at any time $t$.
Namely
\[
Y_t = \max \{  k  \ge 0 \ | \ X_t^{t+k-1}=a_0^{k-1} \} .
\]
The usual abuse of notation $X_t^{t-1}=a_0^{-1}$ means there is no coincidence.
Thus,  one gets
\[
\{  \tau_n > t \} =\left\{ \max_{1\le j \le t}Y_j < n\right\} .
\]
That is, the hitting time problem translates to a maximum problem.
It is well known   
that, under suitable mixing conditions,  the hitting time converges to an exponential law.
The most general result to date \cite{AS11} says that for $\alpha$-mixing systems and every ${\bf a}$
\[
\lim_{n\to\infty} \P\left(  \tau_n > \frac{t}{\theta_q   \mu(a_0^{n-1})}\right) = e^{-t} ,
\]
for some $q=q({\bf a}, \alpha)$ which in general is as large as the memory of the process.
Therefore, the problem is how to compute $\theta_q$ for $q$ large (which also may include to determine the appropriat $q$).
Under certain conditions ($\phi$-mixing) \cite {A06, ACG15} it was shown that $q$ can be replaced by the periodicity of the observed set, which in general is short and makes $\theta$ easier to handle. 
In our case, these mixing  conditions are not verified if the alphabet is infinite.  However we show that actually the periodicity of the observed set can still be used to calculate $\theta$. 

The structure of the paper is the following. In Section 2, we introduce the general form of the regenerative processes we consider and basic properties are derived.
Section 3 is dedicated to the decay of correlations of the model.
In Section 4 we compute the parameters for the different cases we consider.
The first one exhibits different values for the finite and limiting extremal index.
The second one exhibits the geometric distribution where finite and infinite case coincide.
The third one shows a case where the limit of the cluster size is actually a sub-distribution and the limiting extremal index does not exist.
Finally, the fourth case shows a cluster size distribution which fluctuates cyclically and thus the limiting extremal index
also does not exist. In all of them the finite extremal index coincides with the inverse of the finite mean cluster size.
Only in the geometric case the cluster size equals to the sojourn size, which we introduce in the beginning of Section~\ref{sec:EI-and-company}.

 \section{The Model} \label{model}

We consider a general construction of regenerative processes of which the Smith model \cite{S88} is a particular case.
They are discrete time models over a finite or countable alphabet. To simplify, from now on we consider that the alphabet is the set of  positive integers $\Nset$.

Let first $(Z_n)_{n \in \mathbb{Z}}$ be an i.i.d sequence of random variables taking positive integer values, with 
common distribution $p_a=\mathbb{P}(Z_n=a), a \in \Nset$ and finite mean $\mathbb{E}(Z_n) < \infty$.
To each $a\in \Nset$ we also associate a distribution  $q_a= (q_a(k))_{k\in \Nset}$.
The process $(X_n)_{n \in \mathbb{Z}}$ that we are going to consider can be described, informally,
in the following way: take $Z_n$, choose a random number $\xi_n$ with distribution $q_{Z_n}$ independent of everything, and repeate
the symbol $Z_n$ a number $\xi_n$ of times.
The blocks of size $\xi_n$ (filled-up with the symbol $Z_n$) are concatenated to  create the process $(X_n)_{n\in\Zset}$.
A suitable initial condition  turns it into a stationary process if we assume that the mean regeneration time
is finite.
To formalize, define the sequence $(X_n)_{n \in \mathbb{Z}}$ as follows.
Let the auxiliary random variable $\zeta$ have distribution 
\begin{equation} \label{invar}
\P(\zeta=a)= \frac{\sum_{k\ge 1} k q_a(k) p_a}{\nu} .
\end{equation} 
It will be used only as a random shift to make the process stationary. 
To that end, for every $n \ge 1$ and each index  $i$ such that 
$$
\zeta + \underset{j=1}{\overset{n-1}{\sum}}\xi_j  \leq \ i   \ <      \zeta+  \underset{j=1}{\overset{n}{\sum}} \xi_j=: \zeta_n,
$$
set $X_i=Z_n$.
(By convention the sum over an empty set of indexes equals zero.)
This defines $X_i$ for all $i \ge \zeta$.
 
Secondly we define the process for the remaining indexes in a similar way. 
That is, for every $n \ge 0$ and each index  $i$ such that 
$$
\zeta_{-n}:=\zeta - \underset{j=0}{\overset{n}{\sum}}\xi_{-j}  \leq \ i   \ <      \zeta -  \underset{j=0}{\overset{n-1}{\sum}} \xi_{-j} ,
$$
set  $X_i=Z_{-n}$.

The times $(\zeta_n)_{n\in\mathbb{Z}}$ which determine a new choice for a symbol $a$  form a regenerative process. 
The process is positive recurrent with stationary measure $\mu$ if and only if the regeneration time has finite mean.
In our case, this mean is
\begin{equation} \label{mean}
\nu:= \E(\xi_1)=  \E(\E(\xi_1|Z_1))= \sum_{a\in\Nset} p_a \E(q_a)  \ ,
\end{equation}
which we assume to be finite.
By Kac's Lemma one has that the invariant measure of a regeneration 
is equal to the
reciprocal of the above display.
The regenerations are useful  to compute the invariant measure of a measurable set $A$, which will follow from conditioning on the last regeneration time of  $(X_n)_{n\in \Zset}$. 
To simplify the notation, for every $j\in \Zset$, we define the events
$$
R_{j}= \{\exists i\in\Zset  \ | \  \zeta_i=j \}
\quad \text{and } \quad 
W_{j} = R_j \cap \bigcap_{i=j+1}^{0} R_i^c,  \text{for}\  j\le 0,
$$
corresponding, respectively, to the occurrence of a regeneration at time $j$ and that no other regeneration occurs until time 0.
The invariant measure of any measurable set $U$ can be computed partitioning the past according to the $W_{-j}$'s
\begin{equation} \label{ma}
\mu(U) =    \sum_{i=0}^{\infty}  \mu(U| W_{-i} ) \mu (W_{-i} )  .
\end{equation} 
In particular, $U=\{X_0>a\}$ gives the  tail distribution and we put 
\begin{equation} \label{tail}
g_a :=\P(X_0>a)= \sum_{j=a+1}^{\infty} \mu(j) \ .
\end{equation}
These formulae will be used later on with the ad hoc properties of each specific model considered.
Finally, note that by construction, the process is reversible, and as a consequence we have conditional independence of 
consecutive blocks. That is
\[
\P_{R_1}(X_0\in A, X_1\in B)= \P_{R_1}(X_0\in A)  \P_{R_1}( X_1\in B) ,
\]
for $A, B \subseteq \Nset$.

We introduce now four particular examples corresponding to different cases which we are going to study in order to illustrate the finite and limiting behaviour of the extremal index.

\vskip0.3cm

{\bf Basic example: i.i.d.} 
As a first basic example, notice that a sequence of i.i.d. random variables is included in this family of processes with
$q_a(k)= \delta_1(k)$ for all $k$ and every $a$. 
We now come to the specific models we consider in this paper. \\

{\bf Smith's model.} 
The model considered by Smith \cite{S88} 
to show that the limiting extremal index and the limiting reciprocal mean of the cluster size may be different
is defined by setting
\[
q_a(k) = \left\{
\begin{array}{ll}
\frac{a-1}{a}  & {\rm \ for  \ }  k=1 \\
\frac{1}{a}  & {\rm \ for  \ }  k=a+1 \\
0  & {\rm \ otherwise  \ }   \\
\end{array} \right. .
\]
 Thus $\E(q_a)=2$, for all $a\in \Nset$, and hence $\nu=2$. 
 Moreover and in particular,  by (\ref{ma}) we get  $\mu(a)=p_a$. In fact, for  $a\in\mathbb{N}$ 
\begin{align}
\mu(a)&=  
\P(X_0=a,W_0)+\sum_{j=1}^a \P(X_0=a,W_{-j})=\frac12p_a+\sum_{j=1}^a\P(X_0=a,R_{-j}, \cap_{i=-j}^0(R_i)^c)\nonumber\\
&=\frac12p_a+\sum_{j=1}^a\P(R_{-j})\P_{R_{-j}}(X_0=a, \cap_{i=-j}^0(R_i)^c)=\frac12p_a+\sum_{j=1}^a\frac12\frac{p_a}{a}=p_a.
\label{eq:stationary-measure-Smith}
\end{align}

{\bf The block model.}
This model is constructed to present a case where the limiting extremal index does not exist, since the exit probability
fluctuates as  the level $u_n$ diverges. 
The same occurs for the limiting cluster size distribution.
However, the finite parameters are equal.
The model is constructed using any distribution $(p_a)_{a\in \Nset}$ but with deterministic distributions $(q_a)_{a\in \mathbb N}$.
Specifically 
\[
q_a(k) = \left\{
\begin{array}{ll}
{1} & {\rm \ for  \ }  k=a \\
0  & {\rm \ otherwise  \ }   \\
\end{array} \right. .
\]
Hence $\E(q_a)=a$ for all $a\in \Nset$, and we get $\nu=\sum_{a=1}^{\infty} a p_a$, which we assume to be finite in order to have stationarity.
In a similar way to Smith's model, it follows by (\ref{ma}) that $\mu(a)=ap_a/\nu$.
\\

\section{Decay of correlations}

A general argument shows  that a mixing  regenerative process is weak Bernoulli (see for instance the book of P.C. Shields \cite{S96}). Thus, the models presented in this paper are all weak Bernoulli.
In some specific cases, stronger decay of correlations can be computed explicitly.
As an illustration, we are going to compute here the probability of having a regeneration after $n$-steps, given another one was observed in the present time.
Namely,
\[
c_n = \P( R_{n+1} \ | \  R_0 ) ,
\]
in the case of a two symbols process.

{\bf  Morse Code and Fibonacci numbers.} Consider the following case, as a basic example of the block model.
Suppose $p_2=1-p_1$ (and $p_a=0$ for $a\ge 3$). That is, the process only takes values $1$ and $2$. 
Further, for $a=1$ and $a=2$ consider $q_a = \delta_a$. 
Namely,  when $1$ is chosen, it is written once, and when $2$ is chosen, it is written twice.
This model represents the messages that can be written with the Morse code where  only points and traces
are allowed. Thus $c_n$ can be regarded as the probability to write a message of exactly length $n$.
Put $x$ to be the total number of $1$'s and similarly $y$ the total number of $2$'s in this message.
We get
\begin{equation} \label{fibo}
c_n =    \sum_{y=0}^{\lfloor n/2 \rfloor} { x+y \choose y}  p_1^x p_2^y .
\end{equation}
The condition $x+2y=n$ allows to rewrite the above display as
\[
c_n= p_1^n \sum_{y=0}^{\lfloor n/2 \rfloor} { n-y \choose y}   (\frac{p_2}{p_1^2})^y .
\]
At the moment, notice that for the golden ratio $p_1= (-1+\sqrt{5})/2=\phi,$ one gets
\[
c_n=\phi^n F_n, 
\]
where $F_n$ is the $n$-thd
Fibonacci number.
Pascal recurrence leads to re-write the above formula as
\[
c_n = p_1 c_{n-1} +  p_2 c_{n-2} \ .
\]
 A recursive formula conditioning on the previous regeneration could be also invoked to obtain this recursion.
It is classical to obtain the solution of this recursion via roots of its characteristic polynomial
\[
x^2- p_1x-p_2 ,
\]
which, since the different roots are $r_1=1$ and $r_2= p_1-1$, takes the form
\[
c_n=  K_1  1^n + K_2 (p_1-1)^n .
\]
With the initial condition $c_0=1, c_1=p_1$, the constants become
\[
K_1=   \frac{1}{2-p_1} \quad ; \quad K_2=   \frac{1-p_1}{2-p_1}  .
\]
Thus, notice that $K_1=\P(R_0)$ and 
since $c_n$ converges to $K_1$, we get the (exponential) decay of correlations.
Now, it follows easily that the process is $\psi$-mixing with exponential rate function $\phi(n)=(1-p_1)^n$.
For easy reference, we recall the reader that $\psi$ is defined as
\[
\psi(n) = \sup_{A\in \C^u , B\in f^{-(n+u)}\C^v, u,v\in \Nset} \left|    \frac{\P(A \cap B)}{\P(A )  \P(B)} -1 \right| ,
\]
where $f$ is the shift operator.

\vskip0.3cm
{\bf  The Finite Block Model.}
The argument on the example above  can be easily carried on (except for  interpretation (\ref{fibo})) to prove that the process $(X_n)$ considered in this paper over a finite alphabet $\C$
are exponentially $\psi$-mixing.

It is worth noticing that the above methodology captures the eigenvalues of the Perron-Frobenious operator, 
identifies the largest one with modulus equal to 1, the remaining with smaller modulus and also the rate of mixing given by the spectral gap.  \\

{\bf  The Infinite case.} 
The infinite case must be considered with more attention and may be not  $\psi$-mixing.
Consider for instance a probability $(p_a)_{a\in \Nset}$ with no-null entries $p_a$.
We treat first the block model.
So, suppose further that for each positive integer $a$, one has the conditional probability $q_a=\delta_a$.
That is, each time $a$ is chosen in a regeneration, it is repeated deterministically $a$ times.
Thus
\[
\P(X_{a-1}=a |R_0, X_0=a) =1 .
\]
Since $a$ can be as large as we want, the process can not be $\psi$-mixing.

Now consider the Smith's model. Fix $n\in\Nset$ and for $a>n$ take $A_a=\{ X_{-1}\not=a, X_0=a\}$ and $B=\{X_{n+1}=a\}$.
Thus  $\P(A \cap B)/\P(A) \ge 1/a$ while  $\P(B)=p_{a}$, and the ratio of the last two probabilities can not be close to one.

\section{Extremal index \& company}\label{sec:EI-and-company}

\subsection{Definitions}

In this section we present  specific definitions for the family of parameters we are going to consider.
Being one of the main purposes of this paper, we present them for finite observations and then consider their 
asymptotics.
We begin with the  extremal index. 
We introduce first some notation to simplify the expressions.
For a size $q \in \mathbb{N}$ and a level  $a>0$,
let us define the sets $U_a=\left\{X_0 > a\right\}$ and $A_a^{(q)}=\left\{X_0 >a,X_1\leq a, ..., X_q \leq a\right\}$.

\begin{definition} The (finite, $a$-level) extremal index  (up to time $q$) is defined by
$$
 \theta_q(a)
= 
\mathbb{P}(A_a^{(q)}|U_a) . 
$$
\end{definition}
This is the probability of not observing another exceedance (of level $a$) up to time $q$ given that we begin with the observation of an exceedance at time 0.
This  formula was used firstly by O'Brien and then by other authors for the extremal index (see for instance \cite{AFFR17, FFT15, FFT12, O87}).
  The value of $q$ is determined by the observable $U_a$ and the decay of correlation properties of the process. See \cite[equation (15)]{AFFR16} and the discussion preceding it  regarding adequate choices of $q$.
  In general,  the larger is $q$, the more difficult it will be to compute it.
 In the context of hitting times, 
it was shown that, under fast mixing conditions (\cite{A06, ACG15})  $q$ can be taken 
as the (minimum) \emph{periodicity} of the observable 
(also called shortest possible return time, or  shortest possible distance between two observations of $U_a$).
It is given by the positive integer defined as follows
\[
p(U_a) = \inf\{k\ge 1 \ | \    \P(  X_0 >a  , X_k  >a  )>0  \} .
\]
We call  the \emph{escape probability} \cite{AGR18}  to $\theta_q$ 
when taking $q=p(U_a)$. Namely, the escape probability is  $\theta_{p(U_a)}=\P(A_a^{(p(U_a))} \ | \ U_a)$.

An extremal index smaller than one gives rise to a clustering phenomenon. The size of this cluster,
being random, has a distribution with expectation related to the reciprocal of the extremal index.
One must be careful in defining the size of this cluster.
Two different cases are considered here.
The first one is due to the heuristic argument described in the introduction.
It considers the process \emph{starting} from the observable state of interest and counts for how long does it stay in the same state.
The geometric behaviour of this quantity, called \emph{sojourn time}  (under suitable conditions) was proved by Abadi and Vergne \cite{AV09}.
In practice, this situation is usual in physical problems and computational simulation where an initial condition must be imposed.
It also corresponds to the case when some automatic mechanism detects the occurrence of $U_a$
but failures on the mechanism or in the sample itself do not allow to guarantee that the cluster actually
started at this point. To formalize, let
\[
N_a = \sup\{ k \ge 0  \ | \ X_{jp(U_a)}> a , \forall  0 \le j \le k\} +1 ,
\]
the number of consecutive observations of the excedance of $a$.
The $+1$ at the end corresponds to counting the occurrence of the exceedance at time zero, namely $X_0>a$.
(And we set $N_a=0$ if $X_0\le a$).

\begin{definition} The expected \emph{sojourn} is defined by
\[
\E_{U_a}(N_a  ) .
\]
\end{definition}

The second one is due to a natural interpretation of the process as a time series evolution and then considering the beginning of a cluster. That is, when the process \emph{enters} in the observable state of interest. Stationarity lets us fix this entrance at any position in the time scale.
\begin{definition} Let $E_a=\{ X_{-p(U_a)} \le a  ,  X_0 > a \}$ the \emph{entering} to the exceedance to $a$.
We define the mean cluster size  to the expectation of $N_a$
\[
\E_{E_a}(N_a) .
\]
\end{definition}

Note that in the first one we know that at time 0 we have an exceedance but do not know if a cluster of such exceedance could have been initiated earlier, while in the second one we know that the occurrence at time 0 was the beginning of a cluster.

\subsection{Computations}

In this section we proceed to compute the clustering parameters defined in the previous section, to illustrate the already mentioned
different behaviours.

\subsubsection{ The Smith's model}  We are going to consider first the case of an excedance to a level $a$ and then 
the case of hitting a cylinder of at least size $n$.
\pagebreak

\underline{\bf Exceedances} 

Consider an exceedance of a level $a$, and let us compute the clustering parameters associated to this event.
We begin with the  escape probability. Since in our examples two exceedances can occur immediately one after the other, we 
get $p(U_a)=1.$ Thus we compute
\[
\theta_1(a)= \frac{\P(X_0 > a , X_1 \le a ) }{  \P(X_0 > a  )  } . 
\]
Recall that by (\ref{mean}), one has $\nu=2$.
The  denominator is computed using (\ref{tail}) and (\ref{eq:stationary-measure-Smith}) so that $ \P(X_0>a)  =  \sum_{j>a} p_j$  which gives that in this model $  g_a =e_a$.
The numerator follows by noticing there is a regeneration at time 1 and then the future and the past becomes conditionally independent.
That is, first 
$
\{X_0 > a , X_1 \le a \} = \{X_0 > a , R_1, X_1 \le a \} .
$
It follows that
$$
\P(X_0 > a , X_1 \le a ) = \P(R_1)  \P(X_0 > a | R_1 )  \P(X_1\le  a | R_1 ) .
$$
Since the distribution of $X_0$, conditioned to a regeneration at the origin,  is the distribution of $Z_0$
we get
\[
 \P(X_0 > a | R_1 ) = e_a := \sum_{j>a}p_j  ,  \qquad \text{and } \qquad \P(X_1\le  a | R_1 ) = 1-e_a .
\]
We obtain  $\theta_1(a)=(1-e_a)/2$. The limiting extremal index equals 1/2 as stated in Smith's work \cite{S88}. \\

We compute now the mean  cluster size and then the mean sojourn time.
In the first case one can establish the following  equation according  to whether or not one chooses a block of size one
\begin{eqnarray*}
\E_{E_a}(N_a)
&=& \sum_{j>a} \left(1+\E_{R_1}(N^{(1)}_a )\right)\frac{j-1}{j} \frac{p_j}{g_a} + \sum_{j>a}\left(j+1+\E_{R_{j+1}}(N^{(j+1)}_a) \right) \frac{1}{j} \frac{p_j}{g_a}     \\
&=& 2 +\E_{R_0}(N_a) .
\end{eqnarray*}
Here $N^{(j)}_a)$  stands for the cluster size starting to count at $j$ instead of 0.
The second equality follows by stationarity.
 
Now, a  recursive relation can be established for $x=\E_{R_0}(N_a)$. 
Decomposing the future in choosing or not the symbol  $a$, and if so, 
in the length of the first block,  we can equate
\[
x=\sum_{j > a}(1+x )\frac{j-1}{j}  p_j + \sum_{j > a} (j+1+x)\frac{1}{j} p_j+ O  (1-\sum_{j > a}p_j)    . 
\]
Solving this equation one gets $x=2e_a/(1-e_a)$, and therefore   $\E_{E_a}(N_a)=2/(1-g_a)$  which is the reciprocal of the \emph{finite} escape probability.
This holds even when the expectation of the limiting distribution of the cluster size equals one.

For the mean sojourn time we will use (\ref{meansojourn}). 
To that,  we need first the second moment $\E_E(N^2_a)$. 
We write $N_a$ as the length of the first block plus the length of the cluster after the next regeneration
\[
N_a = F+G
\]
where 
\[
F= \inf\{ \xi_j \ | \  \xi_j\ge 1 \}  ,
\]
and
\[
G= \inf\{ \xi_j  \ | \  \xi_j  > F \}  -F .
\]
With this,  since $F$ and $G$ are independent and by stationarity
\begin{equation} \label{secondm}
\E_E(N^2_a) = \E_E(F^2_a) + 2\E_E(F) \E_{R_0}(N_a) +\E_{R_0}(N_a^2) .  
\end{equation}
Direct computations give
\[
E_E(F)=2  , \qquad E_E(F^2_a) =  \frac{1}{e_a}\sum_{j>a} (j+3) p_j = 4,
\]
and
\[
E_{R_0}(N_a)   = \frac{2e_a}{1-e_a}    \qquad \text{and}  \qquad E_{R_0}(N_a^2)=2  \frac{e_a(e_a+1)}{(1-e_a)^2}  \ .
\]
We conclude then that
\[
\E_{U_a}(N_a) = \frac{\theta_1(a)}{2} ( 6+O(e_a))  , 
\]
which converges to 3/2 as $a$ grows while the expected cluster size converges to 2.
\vskip0.3cm

\noindent
\underline{\bf   Hitting to cylinders} \  \vskip0.3cm

Consider  the infinite sequence ${\bf a}=(a,a,a,...)$  consisting only by the symbol $a$.  
For large $n$,  visits  to the first $n$ symbols of {\bf a} are exceedances of the level corresponding to coincidences of the process 
with {\bf a}. 
We are going to compute the extremal index $\theta$  for exceedances of level the corresponding to $n$ 
coincidences and then we consider  the asymptotics on $n$.
The period of $\{X_0^{n-1}=a\} $ is 1.   Thus,  we are going firstly  to obtain the escape probability $\theta_1(n)$.
Namely   
$$\theta_1(n)=\frac{\mathbb{P}(  X_0^{n-1}=a , X_{n} \neq a )}{\mathbb{P}(  X_0^{n-1}=a)} . 
$$
As in the previous case, we consider the numerator and the denominator separately.
By reversibility  and conditioning on the regeneration, the numerator is equal to
\[
\frac{1}{2} \P(X_1\not=a |R_0) \P(X_0^{n-1}=a |R_0) .
\]
The first factor to compute is just $1-p_a$. 
For the second one, put  $p=p_a/a$ and $q=p_a (a-1)/a$.
Since one put immediately after the regeneration, either a block of length 1 or $a+1$, one can construct the recursive equation
\begin{equation} \label{an}
\P( X_0^{n-1}=a|R_0)= q \P(X_0^{n-2}=a|R_0) + p \P(X_0^{n-(a+2)}=a|R_0) ,
\end{equation}
which has characteristic polynomial
\[
x^{a+1} -q x^{a} - p = x^{a} (x-q) - p .
\]
From the last expression it follows that, for $a$ odd, it has two roots $r_1, r_2$ positive and negative  respectively with
$0 \le -r_2 < r_1  < 1$ and thus the solution of the recursion takes the form
\begin{equation} \label{odd}
\P( X_0^{n-1}=a|R_0)= K_1 r_1^n + K_2(n) r^n_2 ,
\end{equation} 
with $K_1$ a constant (on $n$) and $K_2(n)$ a polynomial of degree $a-1$.
For $a$ even, it has only one  root $0 < r_1  < 1$ and so the solution of the recursion takes the form
\begin{equation} \label{even}
\P( X_0^{n-1}=a|R_0)= K_1 r_1^n .
\end{equation} 
In either case,  the leading term is the first one. 
Now we compute the denominator, the stationary measure of $a^n$. We decompose it with respect to the previous occurrence of a regeneration. Namely, it is equal to
\[
\frac{1}{2}\sum_{j=0}^{a} \P( W_j , X_{-j}^{n-1}=a|W_{j}) .
\]
The first term has just been computed. For the remaining term, since there is no regeneration at time 0,  
the first block (the one immediately  after the regeneration at $-j$) has to have  size $a+1$.
Therefore, for $1 \le j\le a$
\[
\P( W_j , X_{-j}^{n-1}=a|W_{j}) = p   \P(  X_{-j+a+1}^{n-1}=a|R_{-j+a+1})
= p   \P(  X_{0}^{n-2+j-a}=a|R_{0})  .
\]
The second equality follows by stationarity. The last expression is the one already obtained.
We conclude that
\begin{eqnarray}
\mu(a^n) 
&=&   \frac{1}{2} \left[   \P( X_{0}^{n-1-j}=a|R_{0})+p \sum_{j=1}^{a} \P( X_{0}^{n-1-j}=a|R_{0}) \right] \\
&=&  \frac{1}{2} K_1 r_1^n  \left[ 1+ \frac{p}{r^a} \sum_{j=0}^{a-1}  r_1^{j} \right]  + o(r_1^n)    .
\end{eqnarray}
A direct calculation using the fact that $r$ is the root of the characteristic polynomial gives
that the factor between brackets is equal to
\[
\frac{1-p_a}{r^a(1-r)}.
\]
Finally, we get
\[
\theta_n^{-1}  \approx \frac{1}{1-r}. 
\]

Now we compute the expectation of the cluster size.
It can be easily derived from the conditional  measure of $a^n$ derived in (\ref{odd}) and (\ref{even})
\[
\E_E(N_a) = \sum_{j=0}^{\infty} \frac{\P(X_0^{n-1+j}=a|R_0) }{\P(X_0^{n-1}=a|R_0  )} =  \sum_{j=0}^{\infty} r^j = \frac{1}{1-r} .
\]
Observe that the sojourn distribution is geometric and
this example shows how the cluster and sojourn size coincide in this  case.

\subsubsection{The block model}  
As before, we consider first the case of exceedences of the level $a$ and then the case of hitting a sequence of at least size $n$. \\

\noindent
\underline{\bf Exceedances} 

\vskip0,3cm 
Consider an exceedance of a level $a$, 
We still have in this case $p(U_a)=1$.
As in the Smith's model,  we compute
$\theta_1(a)= \P(X_0 > a , X_1 \le a )/  \P(X_0 > a  )$. 
Similarly to that case, 
\[
\mu(j)= \frac{jp_j}{\nu} \ ,     \quad \P(X_0>a)=g_a= \frac{\sum_{j>a}jp_j}{\nu} \quad \text{and} \quad \nu=\sum_{j\ge 1}jp_j\ .
\]
For the numerator 
$
\P(X_0 > a , X_1 \le a ) = \mu(R_1) \P_{R_1}(X_0 > a , X_1 \le a ) .
$
It follows that
$$
\P(X_0 > a , X_1 \le a ) 
= \frac{1}{\nu} e_a (1-e_a) ,
$$
where we recall that $e_a=\sum_{j>a}p_j$.
Thus  $\theta_1(a)=e_a(1-e_a)/g_a$, and the limiting extremal index is equal to zero. \\

Now, we compute the mean of the cluster size of exceedances of $a$.
Entering $U_a$  means that a regeneration has just occurred.
Recall that by construction of the process, one puts blocks of length $a$ of level $a$.
The first one is mandatory by the initial condition.
Thus $\E_E(N_a)= \frac{g_a}{e_a}+\E_{R_0}(N_a)$.
The recursion for the last expectation is
\[
x=\sum_{j>a}(j+x)p_j + O(1-e_a),
\]
which gives $x=g_a/(1-e_a)$.
Then, it follows a geometric number of blocks with random size, but larger than $a$.
$$
\E_E(N_a) =  \frac{g_a}{e_a} + \frac{g_a}{e_a} \frac{e_a}{(1-e_a)} = \frac{g_a}{e_a(1-e_a)} , 
$$ 
which is the reciprocal of the escape probability.
Even though, the distribution of $N_a$, as $a$ diverges,  does not even converge to a limiting distribution.
Actually, the cumulative distribution   $\P_E(N_a \le k)$ converges pointwise to zero for all $k$, which means the reciprocal of the mean size
does  not even exist and one can not  compare with the extremal index in the limit. 
Despite of this, they are equal for the finite case.
The reason is clear, there is a mass escape in the distribution of the $N_a$'s, they are not uniformly bounded by an integrable function and the  Dominated Convergence Theorem does not hold.
The lack of tightness is at the core of the construction. 
Taking the limit of the  expectation and not the expectation of the limit should be the recipe  for relating it to the extremal index. 

To compute the mean sojourn time, as in the Smith's model we use formula (\ref{meansojourn}), in the appendix.
Thus we first compute 
\[
  \E_E(F) =  \frac{g_a}{e_a}, \qquad
  \E_E(F^2_a) =  \frac{1}{e_a}\sum_{j>a}  j^2 p_j  . \qquad 
\]  
And further 
\[
 \E_{R_0}(N_a) =  \frac{g_a}{e_a} \frac{ e_a}{1-e_a}, \text{} \qquad
 \E_{R_0}(N_a^2) =  \frac{g_a}{e_a} \frac{ e_a(e_a+1)}{(1-e_a)^2} .
\]
Using again Lemma (\ref{lemma_meansojourn}), we obtain that
\[
\E_{U_a}(N_a )  =  O\left(\frac{\sum_{j>a}  j^2 p_j }{g_a}\right)
\]
which differs from the expected cluster size.\\

\noindent
\underline{\bf   Hitting to cylinders} \  \vskip0.3cm

Take yet  the infinite sequence ${\bf a}=(a,a,a,...)$.  Still in this model $p(U_n)=1$.
Let us compute the escape probability  
\[
\theta_1(n)= \P(X_n\not= a | X_{0}^{n-1}=a) =1- \frac{\mu(a^{n+1})}{\mu(a^{n})} .
\]
It is suffice therefore to compute $\mu(a^n)$.
To do that we condition  in the last occurrence of a  regeneration  of the process before $X_0=a$.
Since the process repeats $a$ times the symbol $a$, this regeneration cannot go further than $a$ coordinates before $0$. Thus
\[
\mu(a^n)=\P(X_{0}^{n-1}=a) = \underset{j=0}{\overset{a-1}{\sum}} \P( X_{-j}^{n-1} =a  |  \rj )  \P(\rj) .
\]
Now,  for $0 \le j \le a-1$, write 
\begin{equation} \label{euclid}
n+j=  \left\lceil \frac{n+j}{a} \right\rceil a- s_{n+j} , \quad   0 \le  s_{n+j}  \le a-1 . 
\end{equation}
By construction of the process
\begin{equation} \label{invan}
\mathbb{P}(X_{-j}^{n-1}=a | R_{-j} ) = 
p_a^{\left\lceil \frac{n+j}{a}\right\rceil} .
\end{equation}
We conclude that 
\[
\mu(a^n)=  \frac{1}{\nu}   \underset{j=0}{\overset{a-1}{\sum}}  p_a^{ \left\lceil \frac{n+j}{a}\right\rceil   } 
= \frac{ p_a^{\left\lceil \frac{n}{a} \right\rceil } }{\nu}  
 \left[ s_n+1 +   (r_n-1)  p_a  \right].
\]
Therefore $\theta_1(n)$ is equal to
\[
1-\frac
 {  s_n +   r_n   p_a   } 
 {  s_n+1 +   (r_n -1)  p_a   } 
 = \frac
 {  1 -   p_a  } 
{  s_n+1 +   (r_n -1)  p_a   }  .
\]
Thus, the extremal index, as a limit, does not exist since $s_n$ runs cyclically between $0$ and $a-1$.

We now  estimate the  mean of the distribution of consecutive observations of 
the target sequence  $a^{n}$.
Set $E=\{ X_{-1} \not =a  ,  X_0^{n-1}=a \}$.
We use again the unconventional Euclidean form (\ref{euclid}) and get
$$
\mu_E(N_n \ge k) = \left\{
    \begin{array}{ll}
        1 & \text{if} \   1 \le  k\leq s_n+1,  \\
        p_a^\ell & \text{if} \ \ell a +s_n+1<k\leq (\ell+1)a + s_n+1, \ \ell\ge 1 .
    \end{array}
\right.
$$
We  conclude that the distribution of $N_n$ does not converge to a limit distribution in $n$.
Further,  since a new block of size $a$ is chosen with probability $p_a$, we can establish the following equation
\[
\mathbb{E}_E(N_n)=  (\mathbb{E}_E(N_n)+a)p_a + (s_n+1)(1-p_a) .
\]
It follows that
\[
\mathbb{E}_E(N_n)=s_n+1  + \frac{a p_a}{1-p_a} .
\]
Now, $s_n$ does not  have limit as $n$ diverges.
Thus $\mathbb{E}_E(N_n)$ does not have a limit in $n$.
However it is easy to verify the identity
\[
\frac{1}{\mathbb{E}_E(N_n)} = \theta_1(n) \ .
\]


Let us consider the mean sojourn time. 
In this case
\[
  \E_E(F) =  s_n+1, \qquad
  \E_E(F^2_a) = (s_n+1)^2 , 
 \] 
 and
 \[
 \E_{R_0}(N_a)     =  a \frac{ p_a}{1-p_a}, \text{} \qquad
 \E_{R_0}(N_a^2) = a  \frac{ p_a(p_a+1)}{(1-p_a)^2} .
\]
Using   (\ref{meansojourn}), one can derive an expression for the mean sojourn time. Even though $a$ is fixed and one consider 
asymptotics in $n$, it is interesting in  particular to consider the case of large $a$ for which $ap_a$ is small.
In that case
\[
\E_{E}(N_a )  \approx s_n+1  \qquad \text{and} \qquad      \E_{U_a}(N_a )  \approx   \frac{s_n}{2}+1  .
\]

\section{Extremal vs. Escape}

O'Brien's formula defines  the extremal index as a function of a suitable  number  $q =o( \mu(A)^{-1})$ 
and then putting  $\theta=\theta_q= \P(A^{(q)}|U)$.
This formula was also obtained independently for the exponential law for hitting/return times in 
\cite{A06, ACG15}.
The precise value of $q$ to be taken depend on the properties of decay of correlations of the process and on the observable itself.
In general, the larger is $q$, the more difficult is to compute $\theta_q$.
The lemma below establishes that in the general model  we considered, and for any observable level or cylinder set $U_a$, any $q=o( \P(U_a|R_0)^{-1})$ can be taken, and thus one can 
chose the minimum possible, which is the period of the observable.

\begin{lemma} \label{equiv}
Consider the regenerative process defined in section \ref{model}. Consider the level $a\in \Nset.$ 
The following inequality holds  for all $q\in\Nset$
\[
\left| 1- \frac{\theta_q}{\theta_1}\right| \le      q \P(U_a|R_0)  . 
\]
\end{lemma} 
\begin{remark} The monotonicity of $\theta_q$ as a function of $q$ and the
 lemma above establish that the $\theta_q$'s are  equivalent in ratio
 for all $q= o(\P(U_a|R_0)^{-1})$.
In particular, this shows that the parameter $\lambda_{U_a}$ in the exponential law of the hitting/return time of $U_{a}$ 
can be replaced by $\theta_1$.
\end{remark}

\begin{proof}[Proof of Lemma \ref{equiv}]
Since $p(U_a)=1$ consider
\[
\theta_1 =   \frac{\P( X_0 > a ,   X_{1}  \le  a )}{ \P( X_0 > a ) }
\quad {\textstyle \ and \ } \quad
\theta_q =   \frac{\P( X_0 > a ,  \cap_{j=1}^{q} X_{j}  \le  a )}{ \P( X_0 > a ) } .
\]
The difference of the  probabilities in the numerators is equal to
\[
 \P(X_0 > a  , X_{1}  \le  a ,   \cup_{j=2}^{q}  X_{j}  > a  )  .
\]
Making a disjoint partition of the union in the above probability as a function of the second excedance
we get that it  is equal to
\[
\P(X_0 > a  , X_{1}  \le  a ) \    
\sum_{j=2}^{q} \P\left(  \bigcap_{i=2}^{j-1} X_{i}  \le  a ,  X_{j}  > a |   X_0 > a  , X_{1} \le a\right)   .
\]
Since there must be a regeneration at time $j$, the leading term can be factorized as
\begin{equation} \label{minor}
\P(  \bigcap_{i=2}^{j-1} X_{i}  \le  a ,   R_j   |   X_0 > a  , X_{1} \le a) 
\P(   X_{j}  > a |   R_j) \ .
\end{equation}
The left-most factor is bounded simply by one. The second one is equal to $\P(   X_{0}  > a |   R_0$, independently of $j$,
by stationarity.
This concludes the proof of the lemma.
\end{proof}

\begin{example}
Consider first the case of exceedances of the level $a$ by the process.
Then $\P( X_{0}  > a |  R_0)$ is equal to $e_a.$
For the case of cylinders $a^n$ one has 
$\P( X_{0}  > a | R_0)= O(K^n)$ for a constant $0<K<1$.
For a general cylinder $a_0^{n-1}$ the same proof holds when changing $\theta_1$ by $\theta_{p(a_0^{n-1})}$.
\end{example}

{\bf Sharpness.} 
Instead of bounding the left-most factor in (\ref{minor}) by one, we can compute it exactly, at least in some cases.
Suppose the $X_i$'s are independent random variables. 
In this case, it becomes equal to $(1-p_a)^{j-2}$. summing up to $q$ we obtain
$1-(1-p_a)^q /p_a$. The denominator cancels with $p_a$ coming from $\mu(   X_{j}  > a |   R_j) $.
For large $a$, one has $1-(1-p_a)^q \approx 1-\exp^{- p_a q}$ which for moderate $q$ is approximated by  $q p_a$.
Thus, we obtain the order of magnitude of the upper bound given by the lemma.
This means that the lemma cannot be improved. Only a better constant may be obtained depending on the ad-hoc properties of the process.

Further, the upper bound for the approximation of $\theta_q$ by $\theta_1$ is almost trivial to compute.
For the exceedances, in both models, $\P(X_0 > a | R_0)=e_a$.
In the case of hitting to  $\{X_0^{n-1}=a\}$, one gets  $\P(X_0^{n-1}=a | R_0)=p_a^{ \lceil n/a \rceil}$ for the block model.
For the Smith model, its exponential decay on $n$  was already computed in (\ref{odd}) and (\ref{even})
.

\section{appendix}

The following lemma establishes a general  tool for computing the expected sojourn time. 

\begin{lemma} 
\label{lemma_meansojourn}
The mean sojourn time verifies
\begin{equation} \label{meansojourn}
 \E_{U_a }(N_a)  
=
\frac{\theta_1(a)}{2} \left(  \E_{E }(N^2_a) +  \E_{E}(N_a)    \right) .
\end{equation}
\end{lemma}

\begin{proof}
Put 
\begin{equation} \label{def}
 \E_{U_a }(N_a)  =   \frac{1}{\P(U_a)}  \sum_{k=1}^{\infty} k \P( \cap_{j=0}^{k-1} X_j > a , X_k \le a) .
 \end{equation} 
 A classical equality for a stationary measure establishes that for all $k$
\begin{equation} \label{rever}
\P( \cap_{j=0}^{k-1} X_j > a , X_k \le a) = \P( X_0 \le a , \cap_{j=1}^{k} X_j > a ) .  
\end{equation} 
Let $Y=\max\{k \ge 1 \ | \  \cap_{j=1}^{k} X_j > a \}$. Then the last sum can be stated as
\[
 \sum_{k=1}^{\infty} k \P( X_0 \le a , Y\ge k ) .
\]
Now notice that, since $Y\ge 1$ and calling $E=\{X_0\le a, X_1>a\}$, one has 
\[
\E(Y^2 | E ) =  \sum_{k=1}^{\infty} (2k-1) \P(  Y \ge k |E ) .
\]
Thus it follows that the right hand side of (\ref{def}) is
\[
\frac{ P(E) }{  \P(U_a)}    \frac{ \E(Y^2 | E ) + \E(Y | E )  }{2} .
\]
By (\ref{rever}) with $k=1$ one gets $E=\{X_0 >a, X_1\le a \}$ and this ends the proof.

\end{proof}

\bibliographystyle{abbrv}
\bibliography{Clustering}

\end{document}